\documentclass[11pt]{amsart}
\usepackage{amsmath,amsthm,amsfonts,amssymb,amscd,graphicx, xypic, enumerate}
\usepackage[all]{xy}
\usepackage{mathrsfs}
\usepackage{xspace}
\usepackage{graphics}
\usepackage{graphicx}
\usepackage[colorlinks=true,hyperindex]{hyperref}

\theoremstyle{plain}

\newtheorem{thm}{Theorem}

\newtheorem{lemma}[thm]{Lemma}
\newtheorem{lem}[thm]{Lemma}
\newtheorem{cor}[thm]{Corollary}
\newtheorem{prop}[thm]{Proposition}

\theoremstyle{definition}

\numberwithin{thm}{section}

\newcommand{\comment}[1]{}

\newcommand{\Hom}{\operatorname{Hom}}
\newcommand{\Ext}{\operatorname{Ext}}

\newcommand{\ds}{\displaystyle}

\newcommand{\arc}{(\textsf{ARC})~}
\newcommand{\garc}{(\textsf{GARC})~}

\renewcommand{\to}{\longrightarrow}

\newcommand{\StHom}{\underline{\operatorname{Hom}}}

\newcommand{\D}{\operatorname{D^b(}H\operatorname{)}}

\newcommand{\Lmod}{\operatorname{\Lambda \textrm{-}mod}}
\newcommand{\Hmod}{\operatorname{H \textrm{-}mod}}
\newcommand{\SHmod}{\operatorname{\widehat{H} \textrm{-}\underline{mod}}}
\newcommand{\SHLmod}{\operatorname{\hat{\Lambda} \textrm{-}\underline{mod}}}
\newcommand{\SLmod}{\operatorname{\Lambda \textrm{-}\underline{mod}}}
\newcommand{\SHNmod}{\operatorname{H_n \textrm{-}\underline{mod}}}
\newcommand{\SBNmod}{\operatorname{B_n \textrm{-}\underline{mod}}}
\newcommand{\SGmod}{\operatorname{\Gamma \textrm{-}\underline{mod}}}

\begin{document}

\title[Self-Extensions over $n$-Symmetric Algerbas]{The Vanishing of Self-Extensions over $n$-Symmetric Algebras of Quasitilted Type}

\author{Maciej Karpicz \ and \ Marju Purin}\address{Faculty of Mathematics and Computer Science, \\Nicolaus Copernicus University\\ ul. Chopina 12/18\\ 87-100 Toru{\'n}}\email{mkarpicz@mat.umk.pl}
\address{Department of Mathematics, Statistics, and Computer Science \\St.\ Olaf College\\ Northfield\\ MN 55057\\ USA}\email{purin@stolaf.edu}

\maketitle

\begin{abstract} A ring $\Lambda$ satisfies the Generalized Auslander-Reiten Condition \garc if for each $\Lambda$-module $M$ with $\Ext^i(M,M\oplus \Lambda)=0$ for all $i>n$  the projective dimension of $M$ is at most $n$. We prove that this condition is satisfied by all $n$-symmetric algebras of quasitilted type -- a~broad class of self-injective algebras where every module is $\nu$-periodic.   \end{abstract}

\section{Introduction}

We study the vanishing of self-extensions of finitely generated modules over a~class of self-injective algebras, namely, the $n$-symmetric algebras of quasitilted type.
We show that this family of algebras satisfies the Generalized Auslander-Reiten Condition.

Recall that a~ring $\Lambda$ is said to satisfy the {\it Generalized Auslander-Reiten Condition} (GARC) if for each $\Lambda$-module $M$ with $\Ext_{\Lambda}^i(M,M \oplus \Lambda) = 0$ for all $i > n$ the projective dimension of $M$ is at most $n$. The condition \garc generalizes the original Auslander-Reiten Condition \arc, the case when $n=0$, which dates back to 1975 \cite{AR}. The Auslander-Reiten Conjecture, still an open question, asserts that all artin algebras satisfy the Auslander-Reiten Condition \cite{AR}.  It is clear that any algebra satisfying (GARC) also satisfies (ARC), but it is known that these conditions are not equivalent.

In this paper we prove that the condition \garc is satisfied by all $n$-symmetric algebras of quasitilted type. This is a~class of self-injective algebras obtained via a~standard construction from the repetitive algebras of the quasitilted algebras. Every $n$-symmetric algebra of quasitilted type have the property that $\nu^n=1$, where $\nu$ denotes the Nakayama automorphism. Our motivation comes from the following two facts. First, the only known examples of algebras that do not satisfy (GARC) are self-injective \cite{E, Schulz}. Second, the only modules for which \arc can fail are the $\nu$-periodic modules \cite{D3}. Hence, one wants a~better understanding of the condition \garc for self-injective algebras in general and $\nu$-periodic algebras in particular.

Recently, Erdmann has examined the condition \garc for weakly symmetric algebras with radical cube zero. She has shown that the condition \garc is related to the (Fg) condition of \cite{EHSST}. In particular, a~weakly symmetric radical cube zero algebra does not satisfy \garc precisely when it is tame and does not satisfy (Fg) \cite{E}. We also mention that the condition \garc has been studied for symmetric algebras in \cite{DP2}. However, the study of $n$-symmetric algebras requires completely different methods. Recently, the terminology \garc originates in the thesis of Diveris \cite{D}.

We now give an outline of the article. In Section 2 we collect some basic lemmas and preliminary results. In Section 3 we prove the condition \garc for $n$-symmetric algebras of quasitilted type by dividing the task into two parts. We first prove the claim for $n$-symmetric algebras of tilted type in Corollary 3.9 and then for $n$-symmetric algebras of canonical type in Corollary 3.11.

\section{Preliminaries}
In this section we collect some background material and preliminary results. We assume throughout that all algebras are finite-dimensional $k$-algebras over an algebraically closed field $k$ and all modules are finitely generated.

We draw our attention to self-injective algebras with the property that every module is $\nu$-periodic. Namely, we will study $n$-symmetric algebras obtained from repetitive algebras of quasitilted algebras via a~standard construction, cf. \cite{LeS, EKS, SY}.

Let $B$ be a~tilted algebra of type $H$, where $H$ is a~hereditary algebra. Denote by $\widehat B$ the repetitive algebra of $B$ and let $\nu_{\widehat B}$ be its Nakayama automorphism. Then the self-injective algebra $B_n=\widehat{B}/\nu_{\widehat B}^n$, where $n\in \mathbb{N}$, has the property $\nu_{B_n}^n=1$, where $\nu_{B}$ is the induced Nakayama automorphism. For this reason we call $B_n$ an $n$-symmetric algebra. Of course, the case $n=1$ gives us the usual symmetric algebra. For simplicity, we omit the subscripts and write $\nu$ for the Nakayama automorphism in each case. We say that the self-injective algebra $B_n$ is an $n$-symmetric algebra of tilted type.

A wider class of $n$-symmetric algebras is formed by the algebras of quasitilted type. We say that an $n$-symmetric algebra $B_n$ is of quasitilted type if it is obtained via a~parallel construction starting with a~quasitilted algebra $B$. We recall that the class of quasitilted algebras consists of the tilted algebras along with the quasitilted algebras of canonical type (endomorphism algebras of tilting objects in hereditary categories whose derived categories coincide with the derived categories of the canonical algebras  \cite{LS}). Accordingly, the class of $n$-symmetric algebras of quasitilted type consists of the $n$-symmetric algebras of tilted type along with the $n$-symmetric algebras of canonical type.
For more background and discussion we refer the reader to \cite{SY}.

We now collect some useful properties of self-injective algebras and of the bounded derived category of a~hereditary algebra.

\subsection{Properties of modules over self-injective algebras} \label{subsec:selfinj}
For any two modules $M$ and $N$ over a~self-injective algebra $\Lambda$, we have the following isomorphisms:

\[ \label{eq:1} \StHom_{\Lambda}(M,N) \cong \Ext_{\Lambda}^1(\Omega^{-1}M,N) \cong \Ext_{\Lambda}^1(M,\Omega N) \tag{1}\]
where $\Omega$ denotes the syzygy operator.

\[\label{eq:2}   \Ext_{\Lambda}^i(M,N) \cong \Ext_{\Lambda}^{i-j}(\Omega^jM,N) \cong \Ext_{\Lambda}^{i-j}(M,\Omega^{-j}N)  \tag{2}\]
This is the degree shifting property for modules over self-injective algebras.

\[ \label{eq:3} \operatorname{D}\Ext_{\Lambda}^1(M,N) \cong \StHom_{\Lambda}(\tau^{-1}N,M) \cong \overline{\Hom}_{\Lambda}(N,\tau M) \tag{3}\]
where $\operatorname{D}(-)=\Hom_{\Lambda}(-,k)$ is the standard duality and $\tau$ denotes the AR translate. This is the Auslander-Reiten formula.

\subsection{Properties of the bounded derived category $\D$} \label{subsec:derived}
If $H$ is a~hereditary algebra of $\D$, then an object is a~bounded complex of $H$-modules. The morphisms are obtained by formally inverting all quasi-isomorphisms in the homotopy category. The indecomposable objects in $\D$ are the stalk complexes with indecomposable stalks. For more details on the bounded derived category we refer the reader to \cite{H}.

We now collect some properties of morphisms in $\D$ that become useful to us later. In what follows, $[i]$ denotes the shift of a~complex by $i$ degrees.
\begin{lem}\label{lem:derived_properties} Let $H$ be a~hereditary algebra and $\D$ the bounded derived category of $\Hmod$. Then for any $M, N \in \Hmod$:
\begin{enumerate}
\item $\Hom_{\D}(M[i],N[i]) = \Hom_{\D}(M,N)$ for any $i \in \mathbb{Z}$;
\item $\Hom_{\D}(M,N[i])= 0$ for $i\neq 0,1$;
\item $\Hom_{\D}(M,N[1])=\Ext^1_{H}(M,N)$;
\item $\Hom_{\D}(M,N[0])=\Hom_{H}(M,N)$.
\end{enumerate}
\end{lem}

\section{$n$-symmetric algebras of quasitilted type}

In this section we prove that the Generalized Auslander--Reiten Condition \garc  holds for all $n$-symmetric algebras of quasitilted type. The $n$-symmetric algebras of quasitilted type consists of two classes: the $n$-symmetric algebras of tilted type and the $n$-symmetric algebras of canonical type.  We consider each class separately.

We begin with a~couple of general observations. First, we note that the condition \garc holds for all self-injective algebras of finite representation type. This follows from the following quick lemma as all non-projective modules over such algebras have periodic resolutions, that is, $\Omega^mM\cong M$ for some integer $m>0$ where $\Omega$ is the syzygy operator.

\begin{lemma} Let $M$ be an $\Omega$-periodic module over a~self-injective algebra $\Lambda$. Then $M$ satisfies the condition \garc.
\end{lemma}
\begin{proof} We have $\Omega^m M\cong M$ for some integer $m>0$. The vanishing condition $\Ext^i(M,M)=0$ for all $i \gg 0$ along with properties from Section \ref{subsec:selfinj} then gives  \begin{align*}\Ext^{im}_{\Lambda}(M,M)& \cong\StHom_{\Lambda}(\Omega^{im}M,M)\\ & \cong \StHom_{\Lambda}(M,M)=0\end{align*} for all $i \gg 0$. The latter is only possible when $M$ is a~projective module.
\end{proof}

More generally, the condition \garc holds for all modules over $n$-symmetric algebras that lie in $\tau$-periodic AR components. This follows, from the lemma above since all such modules have periodic resolutions. Namely, the condition $\nu^n \cong 1$ combined with $\tau$-periodicity, say $\tau^k \cong 1$ where $k\in \mathbb{N}$ , gives us $1 \cong \tau^{nk}\cong \nu^{nk}\Omega^{2nk}\cong \Omega^{2nk}$, that is, we have $\Omega$-periodicity. Here we use the fact that over a~self-injective algebra we have $\tau \cong \nu\Omega^2 \cong \Omega^2\nu$. Since $\tau$-periodic and $\Omega$-periodic modules coincide, we shall refer to them simply as periodic modules.

In light of these comments, we need only examine the condition \garc for modules that are not periodic over $n$-symmetric algebras of infinite representation type.

\subsection{$n$-symmetric algebras of tilted type} We proceed to prove that the condition \garc holds for $n$-symmetric algebras $H_n=H/\nu^n$ that are constructed from a~hereditary algebra $H$.

The AR quiver of a~hereditary artin algebra is well-understood. If $H$ is of infinite representation type, there is a~ unique postprojective component and a~unique preinjective component. All other components are called regular. If $H$ is of tame representation type, the regular components are tubes, that is components of the form $\mathbb{ZA}_{\infty}/(\tau^k)$ for some $k\geq 1$. If $H$ is wild, then the regular components are of type $\mathbb Z A_{\infty}$. See  \cite{H}.

In what follows we need some information about morphisms between modules in the preinjective components and regular components. D. Baer proved the following statement about morphisms between regular modules over wild hereditary algebras in \cite{B}.

\begin{lem} \label{lem:Baer} Let $M$ and $N$ be regular modules over a~wild connected hereditary algebra $H$. Then $\Hom_H(M, \tau^i N) \neq 0$ for all $i\gg 0$.
\end{lem}

We now turn to the preinjective components. The following lemma follows from Proposition 5.6 in Ch. IX \cite{ASS}.

\begin{lem} \label{lem:preinj} Let $H$ be a~hereditary algebra of infinite representation type. Let $M$ and $N$ be a~ preinjective $H$-modules. Then $\Hom_H(\tau^iM,N)\neq 0$ for all $i \gg 0$.
\end{lem}

We now use the above lemmas to gain information about morphisms in the derived category $\D$. Recall that if $H$ is of infinite representation type, then the derived category $\D$ has the following types of AR components: a~ transjective component (obtained from the preinjective and postprojective components of $\Hmod$) and regular components (obtained from the regular components of $\Hmod$).

\begin{prop} \label{prop:hereditary}Let $H$ be a~hereditary algebra of infinite representation type.
\begin{enumerate}
\item Let $M$ be an indecomposable non-periodic regular object in the derived category $\D$. Then we have $\Hom_{\D}(M,\tau^iM)\neq 0$ for all $i\gg 0$.
\item Let $M$ be an indecomposable transjective object in the derived category $\D$. Then we have $\Hom_{\D}(\tau^iM,M)\neq 0$ for all $i\gg 0$.
\end{enumerate}
\end{prop}
\begin{proof} (i) By shifting we may assume that $M$ corresponds to an indecomposable module over the hereditary algebra $H$. Since $M$ is not periodic, $H$ is wild and the result follows immediately from the properties of homomorphisms in $\D$ that were recalled in Subsection~\ref{subsec:derived} and Lemma\autoref{lem:Baer}.

(ii) This follows from a~similar argument and Lemma\autoref{lem:preinj}.
\end{proof}


We now examine homomorphisms over the $n$-symmetric algebra $H_n=\widehat{H}/\nu^n $ obtained from the hereditary algebra $H$ of infinite representation type.

\begin{prop} \label{prop:Hn}Let $M$ be a~non-projective non-periodic module over the algebra $H_n$. Then either $\StHom_{H_n}(M, \tau^i M)\neq0$ or $\StHom_{H_n}(\tau^i M, M)\neq0$ for all $i\gg 0$.
\end{prop}
\begin{proof} First observe that modules over the algebra $H_n$ are the $\nu^n$-orbits of objects over the repetitive algebra $\widehat H$. Given a~module $M$ in $\SHmod$, also write  $M$ for its orbit when viewed as a~module over $H_n$.

We employ Happel's Theorem 4.9 in \cite{H} which shows that there is an equivalence of triangulated categories $\D \cong \SHmod$

There are two cases to consider: the case when $M$ is regular and the case when $M$ is transjective.

In the first case, when $M$ is regular , the statement $\StHom_{H_n}(M,\tau^iM)\neq 0$ for all $i \gg0$ follows immediately if we show that $\StHom_{\widehat H}(M,\tau^iM)\neq 0$ holds for all $i\gg0$. The latter, however, follows from (i) of Proposition~\autoref{prop:hereditary} combined with Happel's Theorem.

Similarly, when $M$ is transjective, the statement $\StHom_{H_n}(\tau^i M,M)\neq 0$ for all $i \gg0$ follows from (ii) of Proposition~\autoref{prop:hereditary}.
\end{proof}

We record the following theorem of Wakamatsu \cite{W}.

\begin{thm} \label{thm:Wakamatsu} Let $B$ be a~tilted algebra from a~hereditary algebra $H$. Then there is an equivalence of categories $\mathbf{F}\colon\ds \SBNmod \to \SHNmod$.
\end{thm}

We prove the following general lemma for stable equivalences between self-injective algebras which we later apply in the setting of $n$ symmetric algebras. We note that a~stable equivalence between self-injective algebras preserves representation type \cite{ARS, Kr}.

\begin{lem} \label{lem:stable}Let $\Lambda$ and $\Gamma$ be two self-injective artin algebras of infinite representation type. Assume that $\mathbf{S}\colon \SLmod \to \SGmod$ is an equivalences of categories. Then we have $\StHom_{\Lambda}(M, \tau^j_{\Lambda} M)=0$ precisely when $\StHom_{\Gamma}(\mathbf{S}(M),\tau^j_{\Gamma}\mathbf{S}(M))=0$.
\end{lem}
\begin{proof} We first note that a~stable equivalence between self-injective algebras of infinite representation type commutes with the AR translate (Section 1, Chapter X in [ARS]).  We employ this fact along with properties of stable equivalence to obtain the following  set of isomorphisms:
\begin{align*}
\StHom_{\Gamma}(\mathbf{S}(M), \tau^j_{\Gamma}(\mathbf{S}(M))) &\cong \StHom_{\Gamma}(\mathbf{S}(M), \mathbf{S}(\tau^j_{\Lambda}(M))) \\
&\cong \StHom_{\Lambda}(M, \tau^j_{\Lambda}(M)) \\
\end{align*}

Therefore, $\StHom_{\Lambda}(M, \tau^j_{\Lambda} M)=0 \iff \StHom_{\Gamma}(\mathbf{S}(M),\tau^j_{\Gamma}\mathbf{S}(M))=0$.
\end{proof}

We now return to the study of the $n$-symmetric algebra $B_n =\widehat{B}/\nu^n$ where $B$ is obtained via a~tilt from a~ hereditary algebra $H$.

\begin{thm} Let $M$ be a~non-projective non-periodic module over an $n$-symmetric algbera $B_n$ of tilted type. Then we have the following.
\begin{enumerate} \item  If $M$ corresponds to a~regular module in $\Hmod$, we have \newline $\Ext_{B_n}^{2ni+1}(M,M)\neq 0$ for all $i \gg 0$.
\item If $M$ corresponds to a~non-regular module in $\Hmod$, we have \newline $\Ext_{B_n}^{2ni}(M,M)\neq 0$ for all $i \gg 0$.
\end{enumerate}
\end{thm}
\begin{proof} (i) Since $B_n$ is a~self-injective algebra we have $\tau\cong\nu \Omega^2$. Moreover, since $B_n$ is $n$-symmetric we have $\tau^{ni}\cong\Omega^{2ni}$ for all integers $i$. Therefore we obtain the isomorphisms
\begin{align*}
\Ext^{2ni+1}_{B_n}(M,M) &\cong\Ext^{1}_{B_n}(\Omega^{2ni}M,M) \\ &\cong\Ext^{1}_{B_n}(\tau^{ni}M,M) \\ &\cong \operatorname{D}\StHom_{B_n}(M,\tau^{ni+1}M)
\end{align*} where the last isomorphism is obtained from the AR formula.

Employing Wakamatsu's Theorem\autoref{thm:Wakamatsu} along with Lemma\autoref{lem:stable}, we see that $\StHom_{B_n}(M,\tau^{ni+1}M)=0$ precisely when $\StHom_{H_n}(\mathbf{F}(M),\tau^{ni+1}\mathbf{F}(M))=0$. We saw, however, in Proposition\autoref{prop:Hn} that $\StHom_{H_n}(\mathbf{F}(M),\tau^{ni+1}\mathbf{F}(M))\neq 0$ for all $i \gg0$, that is, we have $\Ext_{B_n}^{2ni+1}(M,M)\neq 0$ for all $i \gg 0$ as desired.

(ii) The proof of (ii) is similar. This time we use the isomorphisms
\begin{align*}
\Ext^{2ni}_{B_n}(M,M)  & \cong\StHom_{B_n}(\Omega^{2ni}M,M) \\
& \cong \StHom_{H_n}(\tau^{ni}M,M)
\end{align*} and the fact that $\StHom_{H_n}(\tau^{ni}M,M)\neq 0$ for all $i \gg0$ by Proposition\autoref{prop:Hn}.
\end{proof}

We note that part (i) of the above theorem for symmetric algebras is obtained in the proof of Theorem 9.4 of \cite{EKS}.

As an immediate corollary we obtain the following result.

\begin{cor} The Generalized Auslander-Reiten Condition holds for $n$-symmetric algebras of tilted type.
\end{cor}

\subsection{$n$-symmetric algebras of canonical type}

The $n$-symmetric algebras of canonical type consist of the following classes: those of finite representation type (these are, in fact, of tilted type \cite{LP, GL}), tame representation type (these are of tubular type, cf. Theorem 7.5 of \cite{SY}), and wild type. Since the condition \garc holds for periodic AR components, it only remains to consider the $n$-symmetric algebras of wild canonical type.

\begin{thm} Let $\Lambda_n$ be an $n$-symmetric algebra of wild canonical type and let $M$ be a~non-projective and non-periodic  $\Lambda_n$-module. Then we have
 $\Ext^{2ni+1}_{B_n}(M,M) \neq 0$ for all $i \gg0$.
\end{thm}
\begin{proof} Geigle and Lenzing have proved that there is an equivalence of triangulated categories between the derived categories $\operatorname{D^b}(\Lmod)$ of a~wild canonical algebra $\Lambda$ and $\operatorname{D^b}(\rm{coh}(\mathbb{X}))$ of coherent sheaves on a~non-singular weighted projective curve $\mathbb{X}$ of genus $g>1$ \cite{GL}.  Note that each coherent sheaf $\mathcal F \in \rm{coh}(\mathbb X)$ splits into a~direct sum $\mathcal F=\mathcal F_0 \oplus  F$  where $\mathcal F_0$ is a~coherent sheaf  of finite length and  $F$ is a~vector bundle.  The AR quiver of $\Lambda$-modules over a~wild canonical algebra is described in terms of $\rm coh ({\mathbb X})$ in \cite{LM, LP, M}.  The self-injective algebras of wild canonical type have stable AR components of type $\mathbb{Z}A_{\infty}$ and stable tubes \cite{LS}. Any non-periodic non-projective module $M$ over the $n$-symmetric algebra $\Lambda_n$  corresponds to an object in $\operatorname{D^b(\Lambda)}$ which in turn corrresponds to a~vector bundle over $\mathbb X$. For ease of notation, call this vector bundle also $M$.

 Now, Theorem 10.1 of \cite{LP} shows that for a~vector bundle $M$ on $\mathbb X$, we have $\Hom_{\mathbb X}(M, \tau^i_{\mathbb X} M )\neq 0$ for $i \gg 0$. Since  $\operatorname{D^b}(\Lmod)  \cong \operatorname{D^b}(\rm{coh}(\mathbb{X})$ we obtain $\Hom_{\operatorname{D^b}(\Lambda)}(M, \tau^i M) \neq 0$ for all $i \gg 0$. Since $\Lambda$ has finite global dimension, Happel's Theorem 4.9 of \cite{H} gives an equivalence  $\operatorname{D^b}(\Lambda) \cong \SHLmod$. We thus obtain  $\StHom_{\widehat{\Lambda}}(M, \tau^i M) \neq 0$ for all $i \gg 0$ where $M$ denotes the corresponding object in $\SHLmod$. The latter then translates to the corresponding statement for non-projective $\Lambda_n$-modules, that is, we obtain $\StHom_{\Lambda_n}(M,\tau^i M)\neq 0$ for all $i \gg0$.  Using once again the isomorhisms from Section 2.1 and the fact that $\Lambda_n$ is $n$-symmetric, we then obtain $\Ext^{2ni+1}_{\Lambda_n}(M,M) \neq 0$ for all $i \gg0$.
\end{proof}

We note that in the case of symmetric algebras, the above statement follows from the proof of the theorem in Section 7.3 of \cite{LS}.

\begin{cor} The Generalized Auslander-Reiten Condition holds for $n$-symmetric algebras of canonical type.
\end{cor}

 Combining the results from Section 3 gives us our desired theorem:

\begin{thm} The Generalized Auslander-Reiten Condition holds for $n$-symmetric algebras of quasitilted type.
\end{thm}

\section*{Acknowledgements}
The project was supported by the Polish National Science Center grant awarded on the basis of the decision number DEC-2011/01/N/ST1/02064.

\end{document}